\documentclass[11pt,a4paper]{article}

\usepackage{authblk}

\usepackage{mathrsfs}
\usepackage{amsfonts}
\usepackage{amsthm}
\usepackage{authblk}
\usepackage{cite}
\usepackage{amsmath,amssymb,amscd,amsthm}
\usepackage{graphics}
\usepackage{enumerate}
\usepackage[colorlinks=true, linkcolor=magenta, urlcolor=blue, citecolor=blue]{hyperref}
\input xy
\xyoption{all}

\numberwithin{equation}{section}

\usepackage{scalefnt}   

\usepackage{comment}
\usepackage{setspace}

\usepackage[paperwidth=210mm,
paperheight=297mm,
a4paper,
twoside,
top=0.8in, left=0.9in, right=0.9in, bottom=1in]{geometry}

\usepackage[active]{srcltx}

\theoremstyle{definition}
\newtheorem{lemma}{Lemma}[section]
\newtheorem{claim}{Claim}
\newtheorem{prop}[lemma]{Proposition}
\newtheorem{cor}[lemma]{Corollary}

\newtheorem{defi}[lemma]{Definition}
\newtheorem{thm}[lemma]{Theorem}

\newtheorem{rmk}[lemma]{Remark}

\newcommand{\C}{\mathbb{C}}
\newcommand{\N}{\mathbb{N}}
\newcommand{\Z}{\mathbb{Z}}

\usepackage{float}
\restylefloat{figure}

\allowdisplaybreaks

\begin{document}
	\title{Local and Quasi Derivations of the Mirror Heisenberg-Virasoro Algebra
}

    \author{Shun Liu and Dashu Xu\thanks{Corresponding Author: dox@mail.ustc.edu.cn}}

	\date{}
	\maketitle
	\begin{abstract}
		We determine the generalized and local derivations of the mirror Heisenberg-Virasoro algebra. As a by-product, we give a complete characterization of $\delta$-(bi)derivations. 
		\bigskip
        
		
	\end{abstract}
	
	\section{Introduction}
    The mirror Heisenberg-Virasoro algebra is an infinite-dimensional Lie algebra that appears in superconformal field theory \cite{B}.
    It is the semi-direct product of the Virasoro algebra and the twisted Heisenberg algebra.
    Of particular Lie algebraic interest are its intrinsic structure and representation theory.
    Irreducible Harish-Chandra modules and smooth modules were fully classified in \cite{L,TYZ}. 
    Certain non-weight irreducible modules, including Whittaker modules, Cartan-free modules and their tensor products, were investigated in \cite{GMZ}. 
    In \cite{GZ}, the authors considered the tensor product of a module of intermediate series and a highest weight module, leading to a family of irreducible weight modules with infinite-dimensional weight spaces.
    The concept of quasi-Whittaker modules for non-semisimple Lie algebras was recently proposed in \cite{C}, and such modules over the mirror Heisenberg–Virasoro algebra have already been analyzed.
    For structural properties, derivations, automorphism group, and biderivations have been completely determined in \cite{LC,ZC}, while generalized and local derivations are yet to be classified.
    
    In this paper, we focus on the generalized and local derivations of the mirror Heisenberg-Virasoro algebra.
    The notion of generalized derivations of Lie algebras was first introduced by G. F. Leger Jr. and E. M. Luks in \cite{LL}.
    This concept is related to, but not limited to, transposed $\delta$-Poisson algebras \cite{AKS,Bai} and post-Lie algebraic structures \cite{B&D}. 
    Generalized derivations have further been extended to various algebraic frameworks; a comprehensive survey on this topic can be found in \cite{Kay}.
    The concept of local derivations was introduced by R.~V. Kadison \cite{Kad}, D. Larson and A. Sourour \cite{Lar}.
    Local derivations for various algebraic structures have been determined in the extensive literature.
    The study of local derivations in the Lie algebra framework has been well established for finite-dimensional semisimple Lie algebras \cite{AK}, Witt algebras \cite{CZZ}, twisted Heisenberg-Virasoro algebra \cite{BO}, and numerous other Lie algebras.
    In the current paper, we are concerned with the local derivations of the mirror Heisenberg-Virasoro Lie algebra.
    
    In Section \ref{sec2}, we recall some notions that are necessary for the current paper.
    In Section \ref{sec3}, we shall determine all the generalized derivations of the mirror Heisenberg-Virasoro algebra.
    As a by-product, we obtain all the $\delta$-(bi)derivations of the mirror Heisenberg-Virasoro algebra.
    In Section \ref{sec4}, we study local derivations of the mirror Heisenberg-Virasoro algebra.
    
    Now, let us state the main results.
    Denote by $\mathcal{D}$ the mirror Heisenberg-Virasoro algebra.
    Then, it has a $\C$-basis
    $$
    \bigg\{ L_{m},\mathbf{c} \,\bigg|\, m \in \mathbb{Z} \bigg\}
    \cup
    \bigg\{ I_{r},\mathbf{l} \,\bigg|\, r \in \frac{1}{2} + \mathbb{Z} \bigg\},
    $$
    and the non-vanishing Lie brackets are as follows:
    \begin{align*}
    \left[L_m, L_n\right] &= (m-n)L_{m+n}+\delta_{m+n,0}\frac{m^3-m}{12}\mathbf{c},\\
    \left[L_m, I_{r}\right] &= -rI_{m+r},\qquad\qquad
    \left[I_{r}, I_{s}\right] = r\delta_{r+s,0}\mathbf{l},
    \end{align*}
    where $m,n\in\Z$ and $r,s\in \frac12+\Z$.
    Denote by $\mathrm{Inn}(\mathcal{D})$ the space of inner derivations of $\mathcal{D}$, and define $$\mathrm{Ann}(\mathcal{D}):=\left\{f\in\mathrm{End_{\C}(\mathcal{D})}\mid f(\mathcal{D})\subseteq \C\mathbf{c}\oplus\C\mathbf{l}\right\}.$$
    Define an outer derivation $\tau \in\mathfrak{Der}(\mathcal{D})$ by $\tau(L_m)=\tau(\mathbf{c})=0$, $\tau(I_r)=I_r$, and $\tau(\mathbf{l})=2\mathbf{l}$. 
    Let $\mathfrak{GDer}(\mathcal{D})$ be the space of generalized derivations of $\mathcal{D}$.
    The main results of this paper is stated as follows.
    \begin{thm}\label{1.1}
        $\mathfrak{GDer}(\mathcal{D})=\C \tau \oplus\C id_{\mathcal{D}}\oplus\mathrm{Inn}(\mathcal{D})\oplus\mathrm{Ann}(\mathcal{D})$.
    \end{thm}
    \noindent Utilizing this theorem, we obtain all the $\delta$-(bi)derivations of $\mathcal{D}$.
    Suppose $\mathfrak{Der}_{\delta}(\mathcal{D})$ (resp. $\mathfrak{BD}_\delta(\mathcal{D})$) is the space of $\delta$-derivations (resp. $\delta$-biderivations). 
    The following corollaries hold.
    \begin{cor}
        \[
\mathfrak{Der}_{\delta}(\mathfrak{D}) = \left\{
\begin{aligned}
&\operatorname{Inn}(\mathcal{D})\oplus \mathbb{C}\tau\qquad\quad\qquad \delta=1;\\
&\C id_{\mathcal{D}} \qquad\qquad\qquad\quad\;\;\,\,\,\,\delta=\frac{1}{2};  \\
&0 \qquad\qquad\qquad\qquad\quad\;\text{otherwise}.
\end{aligned}
\right.
\] 
    \end{cor}
    \begin{cor} 
        $\mathfrak{BD}_{\delta}(\mathcal{D})\ne 0$ iff $\delta=1$, while $\mathfrak{BD}_{1}(\mathcal{D})$ is determined in \cite[Theorem 4.1]{LC}.
    \end{cor}
Regarding the local derivations, we prove the following theorem.
\begin{thm}
    Every local derivation of $\mathcal{D}$ is a derivation.
\end{thm}

Throughout this paper, we use $\C$, $\C^*$, $\Z$, $\Z^*$, $\Z_+$ and $\N$ to denote the sets of complex numbers, non-zero complex numbers, integers, non-zero integers, non-negative integers and positive integers, respectively.
The vector spaces in this paper are assumed to be over $\C$.
For two sets $A$ and $B$, we use $A\setminus B$ to denote the set $\left\{a\mid a\in A, a\notin B\right\}$. 
\section{Preliminaries}\label{sec2}
Let $\mathfrak{g}$ be a Lie algebra. 
First, we recall some notions related to derivations introduced in \cite{LL,AK}.

\begin{defi} \label{gder}
Let $\Delta(\mathfrak{g})$ denote the set of triples $(f, f', f'')$ with $f, f', f'' \in \operatorname{End}(\mathfrak{g})$ such that
\[
[f(x), y] + [x, f'(y)] = f''([x,y]) \quad \text{for all } x, y \in \mathfrak{g}.
\]
(1)\;The maps that belong to
\[
\mathrm{G}\mathfrak{Der}(\mathfrak{g}) = \left\{ f \in \operatorname{End}(\mathfrak{g}) \mid \exists\, f', f''\in\operatorname{End}(\mathfrak{g})  \text{ such that } (f, f', f'') \in \Delta(\mathfrak{g}) \right\}
\]
are called the \textit{generalized derivations} of $\mathfrak{g}$.

\noindent(2)\;The maps that belong to
\[
\mathrm{Q}\mathfrak{Der}(\mathfrak{g}) = \left\{ f \in \operatorname{End}(\mathfrak{g}) \mid \exists\, f'\in\operatorname{End}(\mathfrak{g})\text{ such that } (f, f, f') \in \Delta(\mathfrak{g}) \right\}
\]
are called the \textit{quasi-derivations} of $\mathfrak{g}$.

\noindent(3)\;For $\delta \in \mathbb{C}$, the maps that belong to 
\[
\mathfrak{Der}_{\delta}(\mathfrak{g}) = \left\{ f \in \operatorname{End}(\mathfrak{g}) \mid (\delta f, \delta f, f) \in \Delta(\mathfrak{g}) \right\}
\]
are called the \textit{$\delta$-derivations} of $\mathfrak{g}$.
The notion of $\delta$-derivations is due to Filippov \cite{F98}. 
In particular, 1-derivations are precisely derivations, and we will simply denote them by $\mathfrak{Der}(\mathfrak{g})$.
We use $\mathrm{Inn}(\mathfrak{g})$ to denote the space of inner derivations.
Denote by
\[
\mathfrak{Der}_{[\delta]}(\mathfrak{g})=\sum_{\delta\in\C}\mathfrak{Der}_{\delta}(\mathfrak{g}).
\]

\noindent (4)\;The \textit{quasi-centroid} of $\mathfrak{g}$ is defined to be
\[
\mathrm{QC}(\mathfrak{g}) = \left\{ f \in \operatorname{End}(\mathfrak{g}) \mid (f, -f, 0) \in \Delta(\mathfrak{g}) \right\}.
\]
\end{defi}

\begin{rmk}
(1)\;Let $Z(\mathfrak{g})$ be the center of $\mathfrak{g}$. Then $f\left(Z(\mathfrak{g})\right) \subseteq Z(\mathfrak{g})$ for each $f \in \mathrm{G}\mathfrak{Der}(\mathfrak{g})$. 

\noindent(2)\;If $\mathfrak{g}=[\mathfrak{g},\mathfrak{g}]$ and $(f,f,f')\in\Delta(\mathfrak{g})$, then $f'$ is uniquely determined by $f$.

\noindent(3)\;The following chain of inclusions holds.
\[
\mathfrak{Der}(\mathfrak{g}) \subseteq \mathfrak{Der}_{[\delta]}(\mathfrak{g}) \subseteq \mathrm{Q}\mathfrak{Der}(\mathfrak{g}) \subseteq \mathrm{G}\mathfrak{Der}(\mathfrak{g}).
\]
\end{rmk}

\noindent In \cite{LL}, Leger and Luks proved the following fact:
\begin{equation}\label{gqq}
\mathrm{G}\mathfrak{Der}(\mathfrak{g})=\mathrm{Q}\mathfrak{Der}(\mathfrak{g})+\mathrm{QC}(\mathfrak{g}).
\end{equation}
Finally, let us recall the concept of local derivations for Lie algebras.
\begin{defi}
    A linear map $f\in\mathrm{End}(\mathfrak{g})$ is called a \textit{local derivation} if for every $x\in\mathfrak{g}$, there exists a derivation $D_x\in\mathfrak{Der}(\mathfrak{g})$ that depends on $x$ such that $f(x)=D_x(x)$.
\end{defi}
\begin{rmk}
While every derivation is a local derivation, there exist Lie algebras that admit local derivations which are not derivations \cite{KJ}.
\end{rmk}

\section{Generalized derivations}\label{sec3}
According to Equation \eqref{gqq},
the first step is to determine all the quasi-derivations.
\begin{thm}\label{dsc of q}
 $\mathrm{Q}\mathfrak{Der}(\mathcal{D})=\C \tau \oplus\C id_{\mathcal{D}}\oplus\mathrm{Inn}(\mathcal{D})\oplus\mathrm{Ann}(\mathcal{D})$. 
 \begin{proof}
 The proof is splitted into two claims.
 \begin{claim}\label{c1}
 $\mathrm{Q}\mathfrak{Der}(\mathcal{D})=\C \tau +\C id_{\mathcal{D}}+\mathrm{Inn}(\mathcal{D})+\mathrm{Ann}(\mathcal{D})$. 
 \end{claim}
Suppose $f,f'\in\mathrm{End}(\mathcal{D})$ such that
\begin{equation}
\label{dq}
[f(x),y]+[x,f(y)]=f'([x,y])
\end{equation}
holds for $x,y\in\mathcal{D}$.
Assume the following general forms for $f$ and $f'$, respectively.
\begin{align*}
      f(L_m)&=\sum_{k \in \Z}A(m,k)L_k+\sum_{k \in \Z}B(m,k+\frac{1}{2} )I_{k+\frac{1}{2}} +C(m)\mathbf{c} +D(m)\mathbf{l}, \\[10pt]
      f(I_{m+\frac{1}{2}})&=\sum_{k \in \Z}X(m+\frac{1}{2} ,k)L_k+\sum_{k \in \Z}Y(m+\frac{1}{2} ,k+\frac{1}{2} )I_{k+\frac{1}{2} }+Z(m+\frac{1}{2} )\mathbf{c} +W(m+\frac{1}{2} )\mathbf{l}, \\[10pt]
      f(\mathbf{c} )&=E\mathbf{c} +F\mathbf{l},\\[10pt]
      f(\mathbf{l})&=G\mathbf{c} +H\mathbf{l};\\[10pt]
      f'(L_m)&=\sum_{k \in \Z}A'(m,k)L_k+\sum_{k \in \Z}B'(m,k+\frac{1}{2} )I_{k+\frac{1}{2}} +C'(m)\mathbf{c} +D'(m)\mathbf{l}, \\[10pt]
      f'(I_{m+\frac{1}{2} })&=\sum_{k \in \Z}X'(m+\frac{1}{2} ,k)L_k+\sum_{k \in \Z}Y'(m+\frac{1}{2}, k+\frac{1}{2} )I_{k+\frac{1}{2} }+Z'(m+\frac{1}{2} )\mathbf{c} +W'(m+\frac{1}{2} )\mathbf{l}, \\[10pt]
      f'(\mathbf{c} )&=\sum_{k \in \Z} P'(k)L_k+\sum_{k \in \Z} Q'(k+\frac{1}{2} )I_{k+\frac{1}{2} }+E'\mathbf{c} +F'\mathbf{l},  \\[10pt]
      f'(\mathbf{l})&=\sum_{k \in \Z}R'(k)L_k+\sum_{k \in \Z} S'(k+\frac{1}{2} )I_{k+\frac{1}{2} }+G'\mathbf{c} +H'\mathbf{l}.
 \end{align*}  
 The proof is divided into two main steps.

Step 1:
In this step, we introduce several quasi-derivations to facilitate the reduction process.
Define a derivation $G:\mathcal{L}\to\mathcal{L}$ in the following way:
\[
G:=[A(0,0)-A(1,1)] \mathrm{ad} \left ( L_0 \right )+ \sum_{k \in \Z^*}\frac{A(0,k)}{k} \mathrm{ad}(L_k)+\sum_{k \in \Z}\frac{2B\left ( 0,k+\frac{1}{2} \right )  }{2k+1} \mathrm{ad} \left ( I_{k+\frac{1}{2}}  \right ).
\]
Then, by direct calculations, we obtain the following equations:
\begin{align*}
    G(L_0)&=\sum_{k \in \Z^*}A( 0,k )L_k+\sum_{k \in \Z}B \left ( 0,k+\frac{1}{2} \right ) I_{k+\frac{1}{2} },\\[10pt]
    G(L_1) &=[ A(1,1)-A(0,0)]L_1+\sum_{k \in \Z^*}\frac{k-1}{k}A(0,k)L_{k+1}+\sum_{k \in \Z}B \left ( 0,k+\frac{1}{2} \right ) I_{k+\frac{3}{2} },\\[10pt]
    G\left ( I_{\frac{1}{2} } \right ) &=-\frac{1}{2}[ A(0,0)-A(1,1)]I_{\frac{1}{2}}-\sum_{k \in \Z^*}\frac{A(0,k)}{2k}I_{k+\frac{1}{2}}+B\left ( 0,-\frac{1}{2} \right )\mathbf{l}.  
\end{align*}

Define a quasi-derivation $H$  as follows:
\[
H:=A(0,0)id_{\mathcal{D}}+ \left [ Y\left ( \frac{1}{2},\frac{1}{2}\right ) -\frac{1}{2}  A\left ( 0,0 \right )-\frac{1}{2} A\left ( 1,1 \right )  \right ] \tau,
\]
where $\tau$ is defined before Theorem \ref{1.1}.
Then, we have the following facts:
\begin{align*}
    H\left(L_0\right)&=A\left( 0,0 \right)L_0,\\[10pt]
    H\left(L_1\right)&=A\left( 0,0 \right)L_1,\\[10pt]
    H\left ( I_{\frac{1}{2} } \right ) &= \left [ Y\left ( \frac{1}{2},\frac{1}{2}\right ) +\frac{1}{2}  A\left ( 0,0 \right )-\frac{1}{2} A\left ( 1,1 \right )  \right ] I_{\frac{1}{2} }.
\end{align*}

Now, perform the replacement:
\[ f \mapsto f-\left(G+H\right).\]
Then, the coefficients of $f$ can be assumed to satisfy
\begin{equation}\label{coef}
A(0,k)=A(1,1)=0, \quad\quad B\left ( 0,k+\frac{1}{2}  \right ) =0, \quad\quad Y\left ( \frac{1}{2}, \frac{1}{2}   \right ) =0
\end{equation}
for all integers $k\in\Z$.

Step 2: 
In this step, 
it will be shown that under the hypothesis of  \eqref{coef}, $f\in\mathrm{Ann}(\mathcal{D})$.

Substituting $(x,y)=(L_{m},L_{n})$ into Equation \eqref{dq}.
Comparing the coefficients of the basis elements $L_k$ gives the equation
\begin{equation}\label{qd1}
   \left ( k-2n \right ) A\left ( m,k-n \right )+ \left ( 2m-k \right ) A\left ( n,k-m \right ) =\left ( m-n \right ) A'\left (m+n,k \right )+\frac{m^3-m}{12} \delta _{m+n,0}P'(k).
\end{equation}
Comparing coefficients of the basis elements $I_k$ yields 
\begin{align}
    &\left ( k-n+\frac{1}{2} \right )B\left( m,k-n+\frac{1}{2} \right)-\left ( k-m+\frac{1}{2} \right )B\left ( n,k-m+\frac{1}{2} \right ) \notag\\ 
    =&\,\left ( m-n \right ) B'\left ( m+n,k+\frac{1}{2} \right )+\frac{m^3-m}{12}\delta _{m+n,0}Q'\left ( k+\frac{1}{2}  \right ). \label{qd2}   
\end{align}
Comparing coefficient of the basis element $\mathbf{c}$ gives
\begin{equation}\label{qd3}
   \left ( n-n^3 \right ) A\left ( m,-n \right ) +\left ( m^3-m \right ) A\left ( n,-m \right ) =12\left ( m-n \right ) C' (m+n) +\left ( m^3-m  \right )\delta_{m+n,0}E'.
\end{equation}
Finally, comparing the coefficient of the basis element $\mathbf{l}$ yields the equation
\begin{equation}\label{qd4}
    12\left ( n-m \right )  D'\left ( m+n \right )=\left ( m^3-m \right )\delta_{m+n,0}F'.
\end{equation}

Substituting $m=0, n\ne 0$ and $\left ( m,n \right ) =\left ( 1,-1 \right ) $ into Equation \eqref{qd4} yields 
\[
D' (m)=0 \quad \text{for all}\quad m \in \Z.
\]
Subsequently, we have $F'=0$.

Now,
we consider Equations \eqref{qd2} and \eqref{qd3}.
In fact,
Equations \eqref{qd2} and \eqref{qd3} are identical to Equations (9) and (10) in \cite[Theorem 16]{KKS}.
We also have 
\[
A(1,1)=A(0,k)=0 \quad \text{for all}\quad k \in \Z.
\]
From the proof of \cite[Theorem 16]{KKS},
we obtain
\[
A(m,n)=A'(m,n)=P'(k)=C'(k)=E'=0 \quad \text{for all}\quad m, n, k \in \Z.
\]

A second substitution, namely $(x,y)=(L_{m},I_{n+\frac{1}{2} })$, produces another system:
\begin{align}
    \left ( k-2m \right ) X\left ( n+\frac{1}{2}, k-m  \right ) =&\left ( n+\frac{1}{2}  \right ) X'\left ( m+n+\frac{1}{2},k  \right ) ;  \label{qd5}
    \\[10pt]
    \left ( k-m+\frac{1}{2}  \right ) Y\left ( n+ \frac{1}{2}, k-m+\frac{1}{2}  \right )=&\left ( n+\frac{1}{2}  \right )Y'\left ( m+n+\frac{1}{2},k+\frac{1}{2}   \right );   \label{qd6} \\[10pt]
    \left ( m-m^3 \right ) X\left ( n+\frac{1}{2}, -m \right ) =&\;12\left ( n+\frac{1}{2}  \right ) Z'\left ( m+n+\frac{1}{2}  \right ); \label{qd7} \\[10pt]
    \left ( n+\frac{1}{2}  \right ) B\left ( m,-n-\frac{1}{2}  \right ) =&\left ( n+\frac{1}{2}  \right ) W'\left ( m+n+\frac{1}{2}  \right ). \label{qd8}
\end{align}

From Equation \eqref{qd8}, we obtain 
\begin{equation}\label{qd8-1}
B\left ( m,-n-\frac{1}{2} \right ) =W'\left ( m+n+\frac{1}{2} \right ).    
\end{equation}
Combined with \eqref{coef} yields
\[
B\left ( m,n+\frac{1}{2} \right )=B \left ( 0,n-m+\frac{1}{2}  \right )  =0\quad \text{for all}\quad m, n \in \Z.
\]
Equation \eqref{qd2} simplifies to
\begin{equation}\label{qd2-1}
     \left ( m-n \right ) B'\left ( m+n,k+\frac{1}{2}  \right )+\frac{m^3-m}{12}\delta_{m+n,0}Q'\left ( k+\frac{1}{2} \right )=0.
\end{equation}
Substituting $(m,n)=(2,-2)$ and $(m,n)=(4,-4)$ into Equation \eqref{qd2-1} gives
\[
    8B'\left ( 0,k+\frac{1}{2}  \right )+Q'\left ( k+\frac{1}{2} \right )=
    8B'\left ( 0,k+\frac{1}{2}  \right )+5Q'\left ( k+\frac{1}{2} \right )=0.
\]
Hence $Q'\left ( k+\frac{1}{2} \right )=B'\left(0,k+\frac{1}{2}\right )=0$ for $k \in \Z$.
By \eqref{qd2-1}, we get
\[
B'\left ( m,n+\frac{1}{2} \right )=0\quad \text{for all}\quad m, n \in \Z.
\]

Taking $m=0$ in Equation  \eqref{qd7},
we obtain
\[
Z' (m+\frac{1}{2})=0 \quad \text{ for all } m \in \Z.
\]
Subsequently, we get
\begin{equation}\label{qd7-1}
X\left(n+\frac{1}{2},m\right)=0 \quad \text{ for  } n \in \Z,\quad m \in \Z\setminus \left \{0, \pm1 \right \}.
\end{equation}
Taking $k=m+n$ in Equation \eqref{qd6},
we obtain
\[
\left ( n+\frac{1}{2}  \right ) Y\left ( n+\frac{1}{2}, n+\frac{1}{2}  \right )= \left ( n+\frac{1}{2}  \right ) Y'\left ( m+n+\frac{1}{2}, m+n+\frac{1}{2}  \right ),
\]
which implies 
\begin{equation}\label{qd6-1}
Y\left ( n+\frac{1}{2}, n+\frac{1}{2}  \right )=  Y'\left ( n+\frac{1}{2}, n+\frac{1}{2}  \right )= Y\left ( \frac{1}{2}, \frac{1}{2}  \right )=0 \quad \text{ for all } n \in \Z.
\end{equation}

 Finally, the substitution $(x,y)=(I_{m+\frac{1}{2} }, I_{n+\frac{1}{2} })$ into Equation \eqref{dq} gives 
\begin{align}
     \left ( m+\frac{1}{2}  \right ) \delta_{m+n+1,0}G' =&\;0;\label{qd11}\\[10pt]
   \left ( m+\frac{1}{2}  \right ) \delta_{m+n+1,0} R'\left ( k \right ) =&\;0;\label{qd9}\\[10pt]
   \left ( m+\frac{1}{2}  \right )\delta_{m+n+1,0}S'\left ( k+\frac{1}{2}  \right )=&\left ( m+\frac{1}{2} \right ) X\left ( n+\frac{1}{2}, k-m \right ) -\left ( n+\frac{1}{2}  \right ) X\left ( m+\frac{1}{2}, k-n \right );\label{qd10}\\[10pt]
   \left ( m+\frac{1}{2}  \right )\delta_{m+n+1,0}H'=&\left ( m+\frac{1}{2} \right ) Y\left ( n+\frac{1}{2}, -m-\frac{1}{2}\right )-\left ( n+\frac{1}{2}  \right ) Y\left ( m+\frac{1}{2}, -n-\frac{1}{2}  \right ).\label{qd12}
\end{align}
It is obvious that $G' =R'\left ( k \right ) =0 $ for $ k\in \Z$.
Substituting $n=-1-m$ into Equation \eqref{qd10} gives
\begin{equation}\label{qd10-5}
  \left ( m+\frac{1}{2} \right ) X\left ( -m-\frac{1}{2}, k-m \right ) +\left ( m+\frac{1}{2}  \right ) X\left ( m+\frac{1}{2}, k+m+1 \right ) 
   =\left ( m+\frac{1}{2}  \right )S'\left ( k+\frac{1}{2}  \right ).  
\end{equation}
Setting $m=0$  in Equation \eqref{qd10-5} and combined with relation \eqref{qd7-1} yield
\[
S'\left(k+\frac{1}{2}\right)=0 \quad \text{for}\quad k \in \Z\setminus \left \{-2, 0, \pm1 \right \}.
\]
Taking $m=8$ in Equation \eqref{qd10-5} together with relation \eqref{qd7-1} gives
\[
S'\left(k+\frac{1}{2}\right)=0 \quad \text{for}\quad k \in \Z\setminus \left \{-10, \pm9, \pm8, 7 \right \}.
\]
Hence, we conclude that
\[
S'\left(k+\frac{1}{2}\right)=0 \quad \text{ for all } k \in \Z.
\]

Equation \eqref{qd10} simplifies to 
\begin{equation}\label{qd10-1}
    \left ( m+\frac{1}{2} \right ) X\left ( n+\frac{1}{2}, k-m \right ) =\left ( n+\frac{1}{2}  \right ) X\left ( m+\frac{1}{2}, k-n \right ).
\end{equation}
Taking $(k,m) =(9n+4,9n+4)$ in Equation \eqref{qd10-1} and  combining with relation \eqref{qd7-1} give
\[
\left ( 9n+4+\frac{1}{2}  \right ) X\left (n+\frac{1}{2},0 \right)=\left ( n+\frac{1}{2}  \right )X\left ( 9n+4+\frac{1}{2},8n+4 \right) =0,
\]
which implies 
\[
X\left(n+\frac{1}{2},0\right)=0 \quad \text{for all}\quad n \in \Z.
\]
Similarly, taking $(k,m) =(9n+4,9n+5)$ and $(k,m) =(9n+4,9n+3)$ in Equation \eqref{qd10-1} and  combining with relation \eqref{qd7-1} yield
\[
X\left(n+\frac{1}{2},\pm1\right)=0 \quad \text{for all}\quad n \in \Z.
\]
Combining these results, we conclude that
\[
X\left(n+\frac{1}{2},m\right)=0 \quad \text{for all}\quad n, m \in \Z.
\]
According to Equation \eqref{qd5},
we also have 
\[
X'\left(n+\frac{1}{2},m\right)=0 \quad \text{for all}\quad n, m \in \Z.
\]

Substituting $n=-1-m$ into Equation \eqref{qd12} yields
\[
 \left ( m+\frac{1}{2}\right )H'=  \left ( m+\frac{1}{2}\right )Y \left ( m+\frac{1}{2},m+\frac{1}{2}\right )+\left ( m+\frac{1}{2}\right )Y\left ( -m-\frac{1}{2},-m-\frac{1}{2}\right ). 
\]
Together with relation \eqref{qd6-1},
this gives $H'=0$.
Equation \eqref{qd12} reduces to 
\begin{equation}\label{qd12-1}
 \left ( m+\frac{1}{2} \right ) Y\left ( n+\frac{1}{2}, -m-\frac{1}{2}\right )=\left ( n+\frac{1}{2}  \right ) Y\left ( m+\frac{1}{2}, -n-\frac{1}{2}  \right ).   
\end{equation}
Setting $m=0$ in Equation \eqref{qd6} yields
\begin{equation}\label{qd12-2}
 \left ( k+\frac{1}{2}  \right ) Y\left ( n+\frac{1}{2} ,k+ \frac{1}{2} \right ) = \left ( n+\frac{1}{2}  \right ) Y'\left ( n+\frac{1}{2} ,k+ \frac{1}{2} \right ).
\end{equation}
Applying \eqref{qd12-1} and \eqref{qd12-2}  to Equation \eqref{qd6},
we obtain 
\begin{equation}\label{qd-100}
\begin{aligned} 
&\left ( m+n+\frac{1}{2}  \right ) \left ( k-m+\frac{1}{2}  \right ) Y\left ( n+\frac{1}{2} ,k-m+ \frac{1}{2} \right ) \\[10pt]
=& \left (m+ n+\frac{1}{2}  \right ) \left ( n+\frac{1}{2}  \right ) Y'\left (m+ n+\frac{1}{2} ,k+ \frac{1}{2} \right )                           \\[10pt]
=& \left ( n+\frac{1}{2}  \right ) \left ( k+\frac{1}{2}  \right ) Y\left (m+ n+\frac{1}{2} ,k+ \frac{1}{2} \right )\\[10pt]
=&-\left ( n+\frac{1}{2}  \right ) \left[ -\left ( k+\frac{1}{2}  \right ) Y\left (m+ n+\frac{1}{2} ,k+ \frac{1}{2} \right ) \right]\\[10pt]
=&-\left ( n+\frac{1}{2}  \right )\left ( m+n+\frac{1}{2}  \right )Y\left ( -k-\frac{1}{2},-m-n-\frac{1}{2}  \right ). 
\end{aligned}
\end{equation}
Setting $k=-n-1$ in Equation \eqref{qd-100} yields
\[
mY\left ( n+\frac{1}{2},-m-n-\frac{1}{2}   \right ) =0,
\]
which implies
\[
Y\left ( n+\frac{1}{2},-m-\frac{1}{2}   \right ) =0 \quad \text{for}\quad m \ne n.
\]
Taking $(m,k)=(1,-n)$ in  Equation \eqref{qd-100} gives
\[
-\left ( n+\frac{3}{2} \right ) \left (n+\frac{1}{2}  \right ) Y\left (n+\frac{1}{2},-n-\frac{1}{2} \right )= -\left (n+\frac{1}{2}  \right )\left ( n+\frac{3}{2} \right )Y\left (n-\frac{1}{2},-n-\frac{3}{2} \right )=0.
\]
Hence we have
\[
Y\left ( n+\frac{1}{2},-n-\frac{1}{2}   \right ) =0 \quad \text{for all}\quad n \in \Z.
\]
Combining these results, we conclude that
\[
Y\left ( n+\frac{1}{2}, m+\frac{1}{2}   \right ) =0\quad \text{for all}\quad n, m \in \Z.
\]
Using this result in  Equation \eqref{qd12-2},
we obtain 
\[
Y'\left ( n+\frac{1}{2}, m+\frac{1}{2}   \right ) =0\quad \text{for all}\quad n, m \in \Z.
\]

To summarize, we have shown that
\[
 \begin{aligned}
      f(L_m)&=C(m)\mathbf{c} +D(m)\mathbf{l} , \\[10pt]
      f\left ( I_{m+\frac{1}{2} } \right ) &=Z\left ( m+\frac{1}{2}  \right ) \mathbf{c} +W\left ( m+\frac{1}{2}  \right )\mathbf{l} , \\[10pt]
      f(\mathbf{c} )&=E\mathbf{c}+F\mathbf{l},\\[10pt]
      f(\mathbf{l} )&=G\mathbf{c}+H\mathbf{l}.
 \end{aligned}
\]
Hence, we have $f \in \operatorname{Ann}(\mathcal{D})$ with the related map $f'=0$.
This completes the proof of Claim \ref{c1}.

\begin{claim}\label{c2}
The sum in Claim \ref{c1} is a direct sum.
\end{claim} 
Suppose for $x\in\mathcal{D}$,
    \begin{equation} \label{dscq1}
    \left [ \mathrm{ad}\left ( \sum_{j \in \Z} a_jL_j+ \sum_{k \in \Z}  b_{k+\frac{1}{2} }I_{k+\frac{1}{2}}\right) +t_1 \tau +t_2id_{\mathcal{D}}+h \right ](x)=0,
     \end{equation}
where $\tau$ is the outer derivations introduced before Theorem \ref{1.1}, and $h \in \operatorname{Ann}(\mathcal{D})$.

Substituting $x=I_{m+\frac{1}{2} }$ into Equation \eqref{dscq1} yields 
\[
-\sum_{j \in \Z} \left (m+\frac{1}{2}\right )a_j  I_{j+m+\frac{1}{2} }-\left (m+\frac{1}{2}\right )b_{-m-\frac{1}{2} }\mathbf{l}+(t_1+t_2)I_{m+\frac{1}{2}}+h\left (I_{m+\frac{1}{2}}  \right ) =0.
\]
Note that $h\left (I_{m+\frac{1}{2}}  \right ) \in \C\mathbf{c}\oplus\C\mathbf{l}$.
Comparing the coefficients of $I_{j+m+\frac{1}{2} }$,
we have $t_1+t_2=0$ and $a_{j}=0$ for $j \in \Z$.

Setting $x=L_{m}$ in Equation \eqref{dscq1} gives 
\[
\sum_{k \in \Z}\left ( k+\frac{1}{2}  \right ) b_{k+\frac{1}{2} } I_{k+m+\frac{1}{2} } +t_2L_m+h\left ( L_m \right )=0.
\]
Comparing the coefficients  of $ I_{k+m+\frac{1}{2} }$ and $L_m$,
we have 
\[
t_2=0\quad \text{and} \quad b_{k+\frac{1}{2} }=0 \quad \text{for } \quad k \in \Z.
\]
It is also obvious that $h=0$.
This completes the proof of Claim \ref{c2}.
\end{proof}
\end{thm}

The above theorem helps to determine all the $\delta$-(bi)derivations of $\mathcal{D}$.  
\begin{cor}\label{t-dd}
\[
\mathfrak{Der}_{\delta}(\mathfrak{D}) = \left\{
\begin{aligned}
&\operatorname{Inn}(\mathcal{D})\oplus \mathbb{C}\tau\qquad\quad\qquad \delta=1;\\
&\C id_{\mathcal{D}} \qquad\qquad\qquad\quad\;\;\,\,\,\,\delta=\frac{1}{2};  \\
&0 \qquad\qquad\qquad\qquad\quad\;\text{otherwise}.
\end{aligned}
\right.
\] 
\end{cor}
\begin{proof}
Let $f$ be a $\delta$-derivation. Then $f$ is certainly a quasi-derivation. 
By Theorem~\ref{dsc of q}, we have the direct sum decomposition
\[
f = f_1 + f_2 + f_3,
\]
where $f_1 \in \C\tau \oplus \operatorname{Inn}(\mathcal{D})$, $f_2 \in \C id_{\mathfrak{D}}$, and $f_3 \in \operatorname{Ann}(\mathcal{D})$.
Note that $\mathcal{D} = [\mathcal{D}, \mathcal{D}]$.
Since $f$ is a $\delta$-derivation, $f = \delta f'$. 
Moreover, by the properties of quasi-derivations, we have $f' = f_1 + 2f_2$. Hence
$
f_1+f_2+f_3 = \delta(f_1+2f_2)
$,
which gives
$
(1-\delta)f_1 + (1-2\delta)f_2 + f_3 = 0
$.
\begin{itemize}
\item If $\delta = 1$, then the equation yields $f_2 = f_3 = 0$. Thus $f = f_1$, and consequently, $\mathfrak{Der}_{1}(\mathcal{D}) = \operatorname{Inn}(\mathcal{D}) \oplus \C\tau$.

\item If $\delta = \tfrac12$, then $f_1 = f_3 = 0$. Hence $f = f_2$ and
$
\mathfrak{Der}_{\frac{1}{2} }(\mathcal{D}) = \C id_{\mathcal{D}}
$.

\item If $\delta \neq 1,\tfrac12$, then $f_1 = f_2 = f_3 = 0$. Therefore $f = 0$, and we obtain
$
\mathfrak{Der}_{\delta}(\mathcal{D}) = 0$.\qedhere
\end{itemize}
\end{proof}
\begin{rmk}
There is a minor inaccuracy in the results for the derivations of the algebra $\mathcal{D}$ in \cite{ZC}.
\end{rmk}

From Corollary \ref{t-dd}, the algebra $\mathcal{D}$ admits no non-zero $\delta$-derivations if $\delta\ne1,\tfrac{1}{2}$.
Next, we consider the $\delta$-biderivations of $\mathcal{D}$.
\begin{defi}\cite[Definition 2.1]{YH}
\label{bider}
Let $\mathfrak{g}$ be a Lie algebra and $\delta\in\C$.
    A bilinear map $f: \mathfrak{g} \times \mathfrak{g}\rightarrow \mathfrak{g}$ is called a $\delta$-\textit{biderivation}, if for all $x, y, z \in \mathfrak{g}$, the following two identities hold:
    \begin{align}
    \label{biderequ1}
        f\left ( \left [ x,y \right ], z \right ) &=\delta\left [ x, f\left (  y,z\right )  \right ] +\delta\left [ f\left ( x,z \right ), y \right],\\
    \label{biderequ2}
       f\left ( x,\left [ y,z \right ]  \right ) &=\delta\left [ f\left ( x,y \right ),z  \right ] +\delta\left[ y,f\left ( x,z \right )  \right ].
    \end{align}
    Denote the space of $\delta$-biderivations by $\mathfrak{BD}_\delta(\mathfrak{g})$.
\end{defi}

\begin{cor}\label{d-bider}
    $\mathfrak{BD}_{\delta}(\mathcal{D})\ne 0$ if and only if $\delta=1$, and $\mathfrak{BD}_{1}(\mathcal{D})$ is determined in \cite[Theorem 4.1]{LC}.
\end{cor}
\begin{proof}
Let $f\in\mathfrak{BD}_\delta(\mathcal{D})$.
Suppose $\delta\ne1,\frac12$.
For $x\in\mathcal{D}$, the unary map $f(x,\cdot):\mathcal{D}\to\mathcal{D}$ is a $\delta$-derivation of $\mathcal{D}$. 
By Corollary \ref{t-dd}, we have $f=0$.
Suppose $\delta=\frac{1}{2}$.
For $x\in\mathcal{D}$, there exists $\lambda_x\in\C$ such that $f(x,\cdot)=\lambda_xid_{\mathcal{D}}$.
Since $[\mathcal{D}, \mathcal{D}]=\mathcal{D}$ and $[\mathcal{D},\mathbf{c}]=0$.
From \cite[Proposition 3.1]{XCK}, we deduce that $0=f(x,c)=\lambda_xc$, leading to $f=0$.
\end{proof}

In the second step, we calculate the quasi-centroid maps of $\mathcal{D}$.
\begin{prop}\label{quasi-centroid}
   $\mathrm{QC}(\mathcal{D}) =\C id_{\mathcal{D}}\oplus\mathrm{Ann}(\mathcal{D})$.
\begin{proof}
Let $g \in \mathrm{QC}(\mathcal{D})$. 
There exists $h\in\mathrm{Ann}(\mathcal{D})$ such that the actions of the quasi‑centroid map $f:= g - h$ on the basis elements are given by
\begin{align*}
    f\left ( L_m \right ) &=\sum_{k \in \Z}A\left ( m,k \right ) L_k + \sum_{k \in \Z}B\left ( m,k+\frac{1}{2}  \right ) I_{k+\frac{1}{2}}, \\[10pt]
    f\left ( I_{m+\frac{1}{2} } \right ) &=\sum_{k \in \Z}X\left ( m+\frac{1}{2} ,k \right ) L_k + \sum_{k+\frac{1}{2}  \in \Z}Y\left ( m+\frac{1}{2} ,k+\frac{1}{2}  \right ) I_{k+\frac{1}{2}}, \\[10pt]
    f\left ( \mathbf{c} \right ) &=A\left ( 0, 0 \right )\mathbf{c} ,  \qquad \qquad \qquad \qquad  f\left ( \mathbf{l} \right ) =A\left ( 0, 0 \right )\mathbf{l} . 
\end{align*}
By definition, for $x,y\in\mathcal{D}$,
we have 
\begin{equation} \label{qcd}
[f(x), y]=[x, f(y)].
\end{equation}

First, substitute $(x,y) = (L_m,L_n)$ into Equation \eqref{qcd} yields the system
\begin{align}
   \left (k-2n \right )A\left ( m,k-n \right )  &=\left (2m-k\right )A\left ( n,k-m \right ), \label{qc-1}  \\[10pt]
   \left ( n-n^3 \right ) A\left ( m,-n \right ) &=\left ( m^3-m \right ) A\left ( n,-m \right ), \label{qc-2}  \\[10pt]
   \left ( n-k-\frac{1}{2} \right ) B\left ( m,k-n+\frac{1}{2} \right )  &=\left ( k-m+\frac{1}{2} \right ) B\left ( n,k-m+\frac{1}{2} \right ). \label{qc-3}
\end{align}
Taking $k=m+n$ in Equation  \eqref{qc-1} gives $(m-n)A(m,m) = (m-n)A(n,n)$, hence
\[
A(m,m) = A(0,0) \quad \text{for all}\quad m\in\Z.
\]
Setting $m=n$ in Equation  \eqref{qc-1} yields $(k-2n)A(n,k-n)=0$, from which we obtain
\[
A(m,n) = 0 \quad \text{whenever } m\neq n.
\]
Thus, for $A(m,n)$ we have
$
A(m,n) = \delta_{m,n}A(0,0)
$.
Taking $m=n$ in Equation \eqref{qc-3} gives 
\[
B\left(m,n+\frac12\right) = 0 \quad \text{for all}\quad m,n\in\Z.
\]

Next, substituting $(x,y) = \left(L_m, I_{n+\frac{1}{2}} \right)$ into  into Equation \eqref{qcd} gives the following equations:
\begin{align}
   \left ( 2m-k \right ) X\left ( n+\frac{1}{2}, k-m  \right ) &=0 \label{qc-4}  \\[10pt]
   \left ( m^3-m \right ) X\left ( n+\frac{1}{2}, -m  \right ) &=0, \label{qc-5}  \\[10pt]
   \left ( k-m+\frac{1}{2}  \right ) Y\left ( n+\frac{1}{2},k-m+\frac{1}{2} \right ) &=\left ( n+\frac{1}{2}  \right )A\left ( m,k-n \right ). \label{qc-6}
\end{align}
By setting $k=1$ in Equation \eqref{qc-4}, we see that
\[
X\left ( m+\frac{1}{2},n \right ) =0 \quad \text{for all}\quad  m,n \in \Z.
\]
Furthermore, 
from \eqref{qc-6}, we obtain
\[
Y\left ( m+\frac{1}{2},n+\frac{1}{2} \right ) =\delta_{m,n}A\left ( 0, 0 \right ) \quad \text{for all }  m,n \in \Z.
\]
Consequently, $f = A(0,0)\,id_{\mathcal{D}}$.
Therefore
$\mathrm{QC}(\mathcal{D}) =\C id_{\mathcal{D}}\oplus\mathrm{Ann}(\mathcal{D})$.
\end{proof}
\end{prop}

Proposition~\ref{quasi-centroid} immediately yields
$
\mathrm{G}\mathfrak{Der}(\mathcal{D})=\mathrm{Q}\mathfrak{Der}(\mathcal{D})
$.
Combining this with Equation~\eqref{gqq}, we conclude that every generalized derivation is thus a quasi-derivation.

\section{Local derivations}\label{sec4}
Suppose $\Delta$ is a local derivation and $ \Delta(L_0)=D_0(L_0)$ for some derivation $D_0$.
By replacing $\Delta$ with $\Delta-D_0$,
we may assume that $\Delta(L_0)=0$.
Fix $m \ne 0$,
since $L_m$ does not produce $\mathbf{l}$ under any derivation, 
we set
\begin{equation}\label{L-1}
\Delta(L_m)=\sum_{n \in \mathbb{Z} }\left ( a_nL_n+b_{n+\frac{1}{2}}I_{n+\frac{1}{2}} \right )+Z_1(m)\mathbf{c},
\end{equation}
where $a_n, b_{n+\frac{1}{2}}, Z_1(m) \in \C$ for $m, n \in \Z$.

For any integer $0 \le i \le m-1$,
define $\overline{i}:=\left \{ i+km \mid k \in \Z \right \} $.
It follows that
$$
\Z=\overline{0} \cup \overline{1}\cup \cdots \cup \overline{m-1}
$$
and Equation \eqref{L-1} can be rewritten as follows:
\begin{equation}\label{L-2}
\Delta(L_m)=\sum_{i  \in E} \sum_{k=s_i}^{t_i}a_{i+km}L_{i+km}+\sum_{i  \in F} \sum_{k=p_i}^{q_i}b_{i+km+\frac{1}{2} }I_{i+km+\frac{1}{2} }+Z_1(m)\mathbf{c},
\end{equation}
where  $E, F \subset \left \{ 0,1,2,\cdots,m-1 \right \}$.

For any $x \in \C^*$,
on the one hand,
since $\Delta\left ( L_0 \right ) =0$,
we have 
\[
\Delta\left ( L_m \right ) =\Delta\left ( L_m +xL_0\right ).
\]
On the other hand, 
by Corollary~\ref{t-dd}, 
since the outer derivation $\tau$ annihilates $L_m+xL_0$, 
there exists an element
\[
\sum_{i  \in E'}\sum_{k=s_i'}^{t_i'}a_{i+km}'L_{i+km}+\sum_{i  \in F'} \sum_{k=p_i'}^{q_i'}b_{i+km+\frac{1}{2} }'I_{i+km+\frac{1}{2} } \in \mathcal{D},
\]
where $E', F' \subset \left \{ 0,1,2,\cdots,m-1 \right \}$, such that 
\begin{align}
&\Delta (L_m+xL_0)  \notag  \\[10pt]
=&\left [\sum_{i  \in E'}\sum_{k=s_i'}^{t_i'}a_{i+km}'L_{i+km}+\sum_{i  \in F'} \sum_{k=p_i'}^{q_i'}b_{i+km+\frac{1}{2} }'I_{i+km+\frac{1}{2} }, L_m+xL_0\right ]  \notag   \\[10pt]
=&\frac{m-m^3}{12} a_{-m}\mathbf{c}+\sum_{i \in E'}\left [\sum_{k=s_i'}^{t_i'}a_{i+km}'(i+(k-1)m)L_{i+(k+1)m}+\sum_{k=s_i'}^{t_i'}a_{i+km}'x(i+km)L_{i+km} \right] \notag   \\[10pt]
&+\sum_{i  \in F'}\sum_{k=p_i'}^{q_i'}b_{i+km+\frac{1}{2} }'\left ( i+km+\frac{1}{2}  \right ) I_{i+(k+1)m+\frac{1}{2}}+\sum_{i  \in F'}\sum_{k=p_i'}^{q_i'}b_{i+km+\frac{1}{2} }'x\left ( i+km+\frac{1}{2}  \right ) I_{i+km+\frac{1}{2}}
 \notag   \\[10pt]
=&\frac{m-m^3}{12} a_{-m}\mathbf{c}+\sum_{i  \in E'}\sum_{k=s_i'}^{t_i'+1} \left[ a_{i+(k-1)m}' \left [ i+(k-2)m  \right ]+ a_{i+km}'x(i+km) \right ] L_{i+km} \notag   \\[10pt]
&+\sum_{i  \in F'}\sum_{k=p_i'}^{q_i'+1}\left [ b_{i+(k-1)m+\frac{1}{2} }' \left ( i+(k-1)m+\frac{1}{2}   \right ) + xb_{i+km+\frac{1}{2} }'\left ( i+km+\frac{1}{2} \right )   \right ] I_{i+km+\frac{1}{2}}, \label{L-3}
\end{align}
where $a_{i+(s_i'-1)m}'=a_{i+(t_i'+1)m}' =b_{i+(p_i'-1)m+\frac{1}{2} }'= b_{i+(q_i'+1)m+\frac{1}{2} }'=0 $.
Finally, comparing the coefficients in   Equations \eqref{L-2} and \eqref{L-3} yields $E=E'$ and $F=F'$.

\begin{lemma}\label{lem-L-1}
Let $\Delta$ be a local derivation and $\Delta(L_0)=0$.
Then $E, F \subset \left \{ 0 \right \}$ in Equation \eqref{L-2}.
\begin{proof}
Assume that $i \ne 0$,
we have $i+km \ne 0$ for all $k \in \Z$.
Comparing the coefficients of $L_{i+km}$ in Equations \eqref{L-2} and \eqref{L-3} yields
\begin{equation}\label{Equ-L-1}
\sum_{k=s_i}^{t_i}a_{i+km}L_{i+km}=\sum_{k=s_i'}^{t_i'+1} \left[ a_{i+(k-1)m}' \left [ i+(k-2)m  \right ]+ a_{i+km}'x(i+km) \right ] L_{i+km}.
\end{equation}
We may assume that $a_{i+s_im} \ne 0$ and $a_{i+t_im} \ne 0$;
then clearly $a_{i+s_i'm}'\ne 0$ and $a_{i+t_i'm}' \ne 0$.
The leading term (i.e., the term with the largest index $k$) on the left hand of \eqref{Equ-L-1} is
\[
a_{i+t_im}L_{i+t_im},
\]
while the leading term  on the right hand of \eqref{Equ-L-1} is 
\[
a_{i+t_i'm}'\left [ i+\left ( t_i'-1 \right )  m\right ]  L_{i+(t_i'+1)m}.
\]
Matching the indices of these leading terms yields $t_i=t_i'+1$.
Similarly, comparing the lowest terms (those with the smallest index $k$) gives $s_i=s_i'$.

Equations \eqref{L-2} now yields the following  relations:
\begin{align*}
    a_{i+s_im} &=xa_{i+s_i m}' \left ( i+s_i m \right ); \\[10pt]
    a_{i+\left ( s_i+1 \right ) m} &=a_{i+s_im}' \left [ i+ \left ( s_i-1 \right ) m \right ]+xa_{i+\left ( s_i+1 \right ) m} ' \left [ i+\left ( s_i+1 \right ) m  \right ]; \\[10pt]
    a_{i+\left ( s_i+2 \right ) m} &=a_{i+\left ( s_i+1 \right ) m} '\left [ i+  s_i m \right ]+xa_{i+\left ( s_i+2 \right ) m} ' \left [ i+\left ( s_i+2 \right ) m  \right ]; \\[10pt]
    &\vdots \\[10pt]
    a_{i+\left ( t_i-2 \right ) m} &=a_{i+\left ( t_i-3 \right ) m} '\left [ i+  \left ( t_i -4 \right ) m \right ]+xa_{i+\left ( t_i-2 \right ) m}'  \left [ i+\left ( t_i-2 \right ) m  \right ] ; \\[10pt]
    a_{i+\left ( t_i-1 \right ) m} &=a_{i+\left ( t_i-2 \right ) m} '\left [ i+  \left ( t_i -3 \right ) m \right ]+xa_{i+\left ( t_i-1 \right ) m}'  \left [ i+\left ( t_i-1 \right ) m  \right ] ; \\[10pt]
    a_{i+t_im} &=a_{i+\left ( t_i-1 \right ) m} '\left [ i+  \left ( t_i -2 \right ) m \right ].
\end{align*}
Since $i+km \ne 0$ for $k \in \Z$,
we may successively eliminate the primed coefficients
\[
 a_{i+s_i m} ',\cdots, a_{i+\left (t_i-1 \right )m'} 
\]
by substitution in this order. 
This yields
\begin{equation}\label{Loc-4}
    a_{i+t_im} +c_1 x^{-1} +c_2 x^{-2} +\cdots+c_{t_i-s_i} x^{-(t_i-s_i)}=0,
\end{equation}
where each $c_j$ are independent of $x$.
We can always find some $x \in \C$ not satisfying Equation \eqref{Loc-4},
which is a contradiction.
Therefore $E \subset \left \{ 0 \right \}$.
Applying the same argument to $I_{i+km}$  gives $F \subset \left \{ 0 \right \}$.
The lemma follows.  
\end{proof}  
\end{lemma}

Lemma \ref{lem-L-1} yields the identity
\begin{equation}\label{Loc-e1}
\Delta(L_m)=\sum_{k\in\mathbb{I}}a_{mk}L_{mk}+\sum_{k\in\mathbb{J}}b_{mk+\frac{1}{2}}I_{mk+\frac{1}{2}}+Z_1(m)\mathbf{c}.
\end{equation}

\begin{lemma}\label{lem-L-2}
With the notation of Equation \eqref{Loc-e1}, 
we have 
$\mathbb{I}=\left \{ 1 \right \}  $, $\mathbb{J} =\emptyset $, and $Z_1(m)=0$.
\begin{proof}
    Combining Equations\eqref{L-2}, \eqref{L-3}, and \eqref{Loc-e1} yields
    \begin{equation}\label{L-6}
        \sum_{k=s}^{t}a_{km}L_{km}=\sum_{k=s'}^{t'+1} \left[ a_{(k-1)m}' \left [ (k-2)m  \right ]+ a_{km}'xkm \right ] L_{km},
    \end{equation}
    where $a_{\left ( s'-1 \right )m } '=a_{\left ( t'+1 \right )m } '=0$.
    We may assume that $a_{sm} \ne 0$ and $a_{tm} \ne 0$; then necessarily 
   $a_{s'm}'\ne 0$ and $a_{t'm}' \ne 0$.
    Clearly, this implies the index bounds
    \[
    s' \le s\le t \le t'+1.
    \]
    Applying the same argument as in Lemma 3.2 in \cite{WGL},
    we obtain $s=t=1$ and $\mathbb{I}=\left \{ 1 \right \}$.
 
Similarly, applying Equations \eqref{L-2}, \eqref{L-3}, and \eqref{Loc-e1} yields
 \begin{equation}\label{L-6-1}
        \sum_{k=p}^{q}b_{km+\frac{1}{2} }I_{km+\frac{1}{2} }=\sum_{k=p'}^{q'+1}\left [ b_{(k-1)m+\frac{1}{2} }' \left ( (k-1)m+\frac{1}{2}   \right ) + xb_{km+\frac{1}{2} }'\left ( km+\frac{1}{2} \right )   \right ] I_{km+\frac{1}{2}},
 \end{equation} 
where $b_{\left ( p'-1 \right )m+\frac{1}{2} } '=b_{\left ( q'+1 \right )m +\frac{1}{2}} '=0$.
We may assume that $b_{pm+\frac{1}{2}},b_{qm+\frac{1}{2}},b_{p'm+\frac{1}{2}}',b_{q'm+\frac{1}{2}}' \ne 0$.
Since $ (k-1)m+\frac{1}{2} \ne 0 $ for all $k \in \Z$,
the highest term on the left hand of Equation \eqref{L-6-1} is
\[
b_{qm+\frac{1}{2} }I_{qm+\frac{1}{2} },
\]
while the highest term on the right-hand side is
\[
\left [ b_{q'm+\frac{1}{2} }' \left ( q'm+\frac{1}{2}   \right )  \right ] I_{\left ( q'+1 \right ) m+\frac{1}{2}}.
\]
Hence we have $q=q'+1$.
Similarly, comparing the lowest terms (those with the smallest index $k$) gives $p=p'$.
Equation \eqref{L-6-1} then yields the following system:
\begin{align*}
    b_{pm+\frac{1}{2} } &= xb_{p m+\frac{1}{2} }' \left ( p m +\frac{1}{2} \right ); \\[10pt]
    b_{\left ( p+1 \right ) m+\frac{1}{2} } &= b_{p m}' \left [ p m+\frac{1}{2}  \right ] + xb_{\left ( p+1 \right ) m+\frac{1}{2} }' \left [ \left ( p+1 \right ) m +\frac{1}{2} \right ]; \\[10pt]
    b_{\left ( p+2 \right ) m+\frac{1}{2} } &= b_{\left(p+1\right)m+\frac{1}{2} }' \left[ \left(p+1\right)m+\frac{1}{2}  \right] + xb_{\left(p+2\right)m+\frac{1}{2} }' \left[ \left( p+2 \right)m +\frac{1}{2} \right]; \\[10pt]
    &\vdots \\[10pt]
    b_{\left ( q-2 \right ) m+\frac{1}{2} } &= b_{\left(q-3\right)m+\frac{1}{2} }' \left[ \left(q-3\right)m +\frac{1}{2} \right] + xb_{\left(q-2\right)m+\frac{1}{2} }' \left[ \left(q-2\right)m+\frac{1}{2}  \right] ; \\[10pt]
    b_{\left ( q-1 \right ) m+\frac{1}{2} } &= b_{\left(q-2\right)m+\frac{1}{2} }' \left[ \left(q-2\right)m +\frac{1}{2} \right] + xb_{\left(q-1\right)m+\frac{1}{2} }' \left[ \left(q-1\right)m +\frac{1}{2} \right] ; \\[10pt]
    b_{q m+\frac{1}{2} } &= b_{\left(q-1\right)m+\frac{1}{2} }' \left[ \left(q-1\right)m \right].
\end{align*}
We now successively eliminate
\[
 b_{p m+\frac{1}{2} } ',\cdots, b_{q m+\frac{1}{2}} '
\]
by substitution in this order. 
This yields
\[
b_{q m+\frac{1}{2} }-b_{(q-1) m+\frac{1}{2} }x^{-1} +b_{(q-2) m+\frac{1}{2} }x^{-2}+\cdots +(-1)^{p-q} b_{p m+\frac{1}{2} }x^{-(p-q)} =0
\]
with $x \in \C^*$.
which is impossible.
Hence $F=\emptyset$.

Since $s=t=1$,
we have $a_{-m}=0$.
Again, using Equations \eqref{L-2}, \eqref{L-3}, and \eqref{Loc-e1}, 
we obtain
 \begin{equation}\label{L-6-2}
        Z_1(m)\mathbf{c}=\frac{m-m^3}{12} a_{-m}\mathbf{c} =0.
    \end{equation}   
This completes the proof of the lemma.  
\end{proof} 
\end{lemma}

Lemma \ref{lem-L-1} and Lemma  \ref{lem-L-2}  show that 
\begin{equation}\label{L-7}
    \Delta(L_m)=a_{m}L_{m}.
\end{equation}

\begin{lemma}\label{lem-L-3}
Let $\Delta$ be a local derivation with $\Delta(L_0)=\Delta(L_1)=0$.
Then $\Delta(L_m)=0$ for all $m \in \Z$.
\begin{proof}
    For a fixed $m \in \Z\setminus \left \{ 0,1 \right \} $.
    By Equation \eqref{L-7} and Corollary \ref{t-dd},
    there exist $a_m \in \Z$ and an element 
    \[
   \sum_{i=r}^{s} a_{i}'L_{i}+\sum_{j=p}^{q}b_{j+\frac{1}{2}}'I_{j+\frac{1}{2}} \in \mathcal{D},
    \]
    such that 
    \begin{equation}\label{lem-L-31}
       a_mL_m=\Delta (L_m)=\Delta (L_m+L_1),
    \end{equation}
    and
    \begin{align}
        &\Delta (L_m+L_1) \notag \\[10pt]
       =&\left [ \sum_{i=r}^{s} a_{i}'L_{i}+\sum_{j=p}^{q}b_{j+\frac{1}{2}}'I_{j+\frac{1}{2}}, L_m+L_1 \right ] \notag\\[10pt]
       =&\sum_{i=r}^{s} a_{i}'(i-m)L_{i+m}   +\sum_{i=r}^{s}a_{i}' (i-1)L_{i+1}  +\frac{m-m^3}{12} a_{-m}'\mathbf{c}\notag\\[10pt]
       &+\sum_{j=p}^{q}b_{j+\frac{1}{2}}'\left ( j+ \frac{1}{2}\right ) I_{j+m+\frac{1}{2}} +\sum_{j=p}^{q}b_{j+\frac{1}{2}}'\left ( j+ \frac{1}{2}\right ) I_{j+1+\frac{1}{2}} \label{lem-L-32}.
    \end{align}
Assume, for contradiction, that
\[
b_{p+\frac{1}{2}}', b_{q+\frac{1}{2}}' \ne 0.
\]
For $m>1$,
the coefficients of $I_{q+m+\frac{1}{2}}$ in Equation \eqref{lem-L-32} is
\[
b_{q+\frac{1}{2}}'\left ( q+\frac{1}{2} \right ),
\]
whereas in Equation \eqref{lem-L-31} it is zero, a contradiction.
For $m<0$,
the coefficients of $I_{q+1+\frac{1}{2}}$ in Equation \eqref{lem-L-32} is
\[
b_{q+\frac{1}{2}}'\left ( q+\frac{1}{2} \right ),
\]
while it is zero in Equation \eqref{lem-L-31}, again a contradiction.
Therefore,
we have 
\begin{equation}\label{lem-L-32-1}
\sum_{j=p}^{q} b_{j+\frac{1}{2}}' I_{j+\frac{1}{2}} = 0.
\end{equation}

Comparing the coefficients  of $L_k$ in Equations  \eqref{lem-L-31} and  \eqref{lem-L-32} yields 
\begin{equation}\label{lem-L-33}
    a_mL_m=\sum_{i=r}^{s} a_{i}'(i-m)L_{m+i}+\sum_{i=r}^{s} a_{i}'(i-1)L_{i+1}.
\end{equation}
Assume, for contradiction, that $a_m \ne 0$. 
Combining with Equations \eqref{lem-L-32} and \eqref{lem-L-32-1},
we have 
\[
\sum_{i=r}^{s}  a_i'  L_i \ne 0.
\]
It is obvious that $a_r', a_s' \neq 0$. 
Suppose $m > 1$.
Then, Equation \eqref{lem-L-33} can be rewritten as
\[
a_m  L_m =  a_r' (r-1)  L_{r+1} + \cdots +  a_s' (s-m)  L_{s+m}.
\]
We distinguish three cases. 
In the first case, if $a_r'(r-1) = a_s'(s-m) = 0$, then $r=1$ and $s=m$. Substituting $(r,s)=(1,m)$ into Equation \eqref{lem-L-33} gives
\[
a_m L_m = a_{m-1}'(m-2) L_m - a_{m-1}'L_{2m-1},
\]
which yields a contradiction since both $a_{m-1}'=0$ and $a_{m-1}'(m-2)\ne0$ are required. 
In the second case, 
if $a_r'(r-1) L_{r+1} = a_m L_m$ and $a_s'(s-m) L_{s+m}=0$, then $r=m-1$ and $s=m$.
It follows that $a_{m-1}', a_m' \ne 0$. 
Then Equation \eqref{lem-L-33}  becomes 
\[
a_m L_m = -a_{m-1}' L_{2m-1} + a_{m-1}'(m-2) L_m + a_m'(m-1) L_{m+1}, 
\]
which is impossible. 
In the third case, if $a_r'(r-1) L_{r+1}=0$ and  $a_s'(s-m) L_{s+m}=a_m L_m$, then $r=1$ and $s=0$, 
which is a contradiction.
Thus, our assumption $a_m \ne 0$ fails, and hence $a_m=0$ for all $m>1$.
The proof for the case $m<0$ is similar, completing the lemma.
\end{proof}
\end{lemma}

Using the fact that $\Delta(L_1)=a_1L_1$, 
set $\Delta_1=\Delta+a_1\operatorname{ad}(L_0)$. 
Then $\Delta_1$ is a local derivation with $\Delta_1(L_0)=\Delta_1(L_1)=0$. Lemma~\ref{lem-L-3} then implies $\Delta_1(L_m)=0$ for every $m\in\mathbb Z$.

\begin{lemma}\label{lem-L-4}
Let $\Delta$ be a local derivation such that $\Delta(L_m)=0$ for all $m\in\mathbb Z$. Then, there exists a constant $k\in\C$ such that
\[
\Delta(I_{m+\frac12})=kI_{m+\frac12}
\quad\text{for all}\quad m\in\mathbb Z.
\]
\begin{proof}
First, we prove that $\Delta(I_{m+\frac12})\in\mathbb C I_{m+\frac12}$.
From the definition of local derivations and the concrete forms of the derivations of $\mathcal{D}$, we have
\begin{equation}\label{lem-L-4-0}
\Delta(I_{m+\frac12})
=
\sum_{i=s_m}^{t_m} d_{i+\frac12} I_{i+\frac12}
+ f_m\mathbf l.
\end{equation}

Choose $p<0$ and $q>0$ such that $p+m<s_m$ and $q+m>t_m$. For $I_{m+\frac12}+L_p+L_q\in\mathcal{D}$, on the one hand,
\[
\Delta(I_{m+\frac12}+L_p+L_q)
=
\Delta(I_{m+\frac12})
=
\sum_{i=s_m}^{t_m} d_{i+\frac12} I_{i+\frac12}
+ f_m\mathbf l.
\]
On the other hand, there exist $d'\in\mathbb C$ and an element
\[
\sum_{i\in\mathbb I} a_i'L_i+\sum_{j\in\mathbb J} b_{j+\frac12}'I_{j+\frac12}\in\mathcal D
\]
such that
\[
\Delta(I_{m+\frac12}+L_p+L_q)
=
\left[
\sum_{i\in\mathbb I} a_i'L_i+\sum_{j\in\mathbb J} b_{j+\frac12}'I_{j+\frac12},
I_{m+\frac12}+L_p+L_q
\right]
+d'\tau(I_{m+\frac12}+L_p+L_q),
\]
where $\tau$ is the outer derivation defined in Corollary \ref{t-dd}. Hence, we have
\begin{equation}\label{lem-4-0}
\sum_{i=s_m}^{t_m} d_{i+\frac12} I_{i+\frac12}+f_m\mathbf l
=
\left[
\sum_{i\in\mathbb I} a_i'L_i+\sum_{j\in\mathbb J} b_{j+\frac12}'I_{j+\frac12},
I_{m+\frac12}+L_p+L_q
\right]
+d'I_{m+\frac12}.
\end{equation}
It is clear that $\sum_{i\in\mathbb I} a_i'L_i=a'(L_p+L_q)$ for some $a'\in\mathbb C$. Then
\begin{align}
\sum_{i=s_m}^{t_m} d_{i+\frac12} I_{i+\frac12}+f_m\mathbf l
={}&-\left(m+\frac12\right)b_{-m-\frac12}'\mathbf l
-\left(m+\frac12\right)a'I_{p+m+\frac12}
-\left(m+\frac12\right)a'I_{q+m+\frac12} \notag\\[6pt]
&+\sum_{j\in\mathbb J} b_{j+\frac12}'\left(j+\frac12\right) I_{p+j+\frac12}
+\sum_{j\in\mathbb J} b_{j+\frac12}'\left(j+\frac12\right) I_{q+j+\frac12}
+d'I_{m+\frac12}. \label{lem-4-1}
\end{align}

Now, we claim that $\mathbb J\subset\{m\}$ in \eqref{lem-4-1}. Let
\[
s'=\min\{ j\in\mathbb J\mid b_{j+\frac12}'\neq0\},
\qquad
t'=\max\{ j\in\mathbb J\mid b_{j+\frac12}'\neq0\}.
\]
If $s'<m$, then $p+s'<p+m<m$, hence in the right-hand side of \eqref{lem-4-1}, the coefficient of $I_{p+s'+\frac12}$ is
\[
b_{s'+\frac12}'\left(s'+\frac12\right).
\]
Since $p+s'<p+m<s_m$, this is a contradiction. 
If $t'>m$, then $q+t'>q+m>m$, and the right-hand side of \eqref{lem-4-1} contains the nonzero term
\[
b_{t'+\frac12}'\left(t'+\frac12\right) I_{q+t'+\frac12}.
\]
Since $q+t'>q+m>t_m$, this is again a contradiction. Thus, we have $\mathbb J\subset\{m\}$.

Moreover, comparing the coefficients of $\mathbf l$ yields
\[
f_m=-\left(m+\frac12\right)b_{-m-\frac12}'=0.
\]
Hence, Equation \eqref{lem-4-1} reduces to
\begin{align}
\sum_{i=s_m}^{t_m} d_{i+\frac12} I_{i+\frac12}
={}&-\left(m+\frac12\right)a'I_{p+m+\frac12}
-\left(m+\frac12\right)a'I_{q+m+\frac12} \notag\\[6pt]
&+b_{m+\frac12}'\left(m+\frac12\right) I_{p+m+\frac12}
+b_{m+\frac12}'\left(m+\frac12\right) I_{q+m+\frac12}
+d'I_{m+\frac12}. \label{lem-4-2}
\end{align}
Since $p+m<s_m$, $q+m>t_m$ and $p+m<m<q+m$, we obtain $b_{m+\frac12}'=a'$. Therefore, we have
\[
\Delta(I_{m+\frac12})
=
\sum_{i=s_m}^{t_m} d_{i+\frac12} I_{i+\frac12}
=
d'I_{m+\frac12},
\]
which implies $\Delta(I_{m+\frac12})\in\mathbb C I_{m+\frac12}$.

Next, write
\[
\Delta(I_{m+\frac12})=k_{m+\frac12} I_{m+\frac12}.
\]
On one hand,
\begin{equation}\label{lem-4-5}
\Delta(L_1+L_{-1}+I_{m+\frac12}+I_{m-\frac12})
=
\Delta(I_{m+\frac12}+I_{m-\frac12})
=
k_{m+\frac12}I_{m+\frac12}
+k_{m-\frac12}I_{m-\frac12}.
\end{equation}
On the other hand, there exist $k'\in\mathbb C$ and an element
\[
\sum_{i\in\mathbb I} a_i''L_i+\sum_{j\in\mathbb J} b_{j+\frac12}''I_{j+\frac12}\in\mathcal D
\]
such that
\begin{align}
\Delta(L_1+L_{-1}+I_{m+\frac12}+I_{m-\frac12})
=&
\left[
\sum_{i\in\mathbb I} a_i''L_i+\sum_{j\in\mathbb J} b_{j+\frac12}''I_{j+\frac12},
L_1+L_{-1}+I_{m+\frac12}+I_{m-\frac12}
\right] \notag\\[6pt]
&+k'\tau(L_1+L_{-1}+I_{m+\frac12}+I_{m-\frac12}). \label{lem-4-6}
\end{align}
Clearly $\sum_{i\in\mathbb I} a_i''L_i=a''(L_1+L_{-1})$ for some $a''\in\mathbb C$. Combining \eqref{lem-4-5} and \eqref{lem-4-6}, we deduce that
\begin{align}
k_{m+\frac12}I_{m+\frac12}+k_{m-\frac12}I_{m-\frac12}
={}&-\left(m+\frac12\right)a''I_{m+\frac32}
-\left(m+\frac12\right)a''I_{m-\frac12}
-\left(m-\frac12\right)a''I_{m+\frac12} \notag\\[6pt]
&-\left(m-\frac12\right)a''I_{m-\frac32}
+\sum_{j\in\mathbb J} b_{j+\frac12}''\left(j+\frac12\right) I_{j+\frac32}
+\sum_{j\in\mathbb J} b_{j+\frac12}''\left(j+\frac12\right) I_{j-\frac12} \notag\\[6pt]
&+k'(I_{m+\frac12}+I_{m-\frac12})
-\left(m+\frac12\right)b_{-(m+\frac12)}''\mathbf l
-\left(m-\frac12\right)b_{-(m-\frac12)}''\mathbf l. \label{lem-4-7}
\end{align}
Let
\[
p''=\min\{j\mid j\in\mathbb J,\ b_{j+\frac12}''\neq0\},
\qquad
q''=\max\{j\mid j\in\mathbb J,\ b_{j+\frac12}''\neq0\}.
\]
If $p''<m-1$, then $p''-\frac12<m-\frac32<m-\frac12$, and the right-hand side of \eqref{lem-4-7} contains the nonzero term
\[
b_{p''+\frac12}''\left(p''+\frac12\right) I_{p''-\frac12},
\]
a contradiction. If $q''>m$, then $q''+\frac32>m+\frac32>m+\frac12$, and the right-hand side of \eqref{lem-4-7} contains the nonzero term
\[
b_{q''+\frac12}''\left(q''+\frac12\right) I_{q''+\frac32},
\]
again a contradiction. Hence $\mathbb J\subset\{m-1,m\}$. Comparing the coefficients of $I_{m+\frac32}$ and $I_{m-\frac32}$ in \eqref{lem-4-7}, we get
\[
b_{m-\frac12}''=b_{m+\frac12}''=a''.
\]
Then \eqref{lem-4-5} and \eqref{lem-4-6} imply
\[
k_{m+\frac12}=k'=k_{m-\frac12}.
\]
The proof is complete.
\end{proof}
\end{lemma}
According to Corollary \ref{t-dd},
it is obvious that $\Delta (\mathbf{c} )=0$ for any local derivation $\Delta$.
\begin{lemma}\label{lemm-6}
Let $\Delta$ be a local derivation such that $\Delta(L_m)=\Delta(I_{m+\frac12})=0$ for all $m\in\mathbb Z$. Then
\[
\Delta(\mathbf{l})=0.
\]
\begin{proof}
Fix $m\in\mathbb Z$. We have
\[
\Delta(I_{m+\frac12})=kI_{m+\frac12}=0,
\]
and hence $k=0$.

By Corollary \ref{t-dd}, there exists $r\in\mathbb C$ such that $\Delta(\mathbf{l})=r\mathbf{l}$. On the one hand,
\begin{equation}\label{lemm-6-1}
\Delta(L_1+L_{-1}+I_{m+\frac12}+I_{m-\frac12}+\mathbf{l})=r\mathbf{l}.
\end{equation}
On the other hand, there exist $k'\in\mathbb C$ and an element
\[
\sum_{i\in\mathbb I} a_i''L_i+\sum_{j\in\mathbb J} b_{j+\frac12}''I_{j+\frac12}\in\mathcal D
\]
such that
\begin{align}
\Delta(L_1+L_{-1}+I_{m+\frac12}+I_{m-\frac12}+\mathbf{l})
={}&
\left[
\sum_{i\in\mathbb I} a_i''L_i+\sum_{j\in\mathbb J} b_{j+\frac12}''I_{j+\frac12},
L_1+L_{-1}+I_{m+\frac12}+I_{m-\frac12}+\mathbf{l}
\right] \notag\\[6pt]
&+k'\tau(L_1+L_{-1}+I_{m+\frac12}+I_{m-\frac12}+\mathbf{l}). \label{lemm-6-2}
\end{align}
By an argument analogous to that used in deriving Equation \eqref{lem-4-6}, we obtain
\[
r=2k'=2k=0.
\]
The proof is complete.
\end{proof}
\end{lemma}


\begin{thm}\label{thm-4-1}
Every local derivation of $\mathcal{D}$ is a derivation.
\begin{proof}
Let $\Delta$ be a local derivation of $\mathcal{D}$. We can always find a derivation $D_0$ such that $\Delta(L_0)=D_0(L_0)$. Set $\Delta_1=\Delta-D_0$. Then $\Delta_1(L_0)=0$.

By Equation \eqref{L-7}, we have $\Delta_1(L_1)=a_1L_1$. Set
\[
\Delta_2=\Delta_1+a_1\mathrm{ad}(L_0).
\]
Then $\Delta_2(L_0)=\Delta_2(L_1)=0$. By Lemma \ref{lem-L-3}, we have $\Delta_2(L_m)=0$ for all $m\in\mathbb Z$. According to Lemma \ref{lem-L-4}, there exists a constant $k$ such that
\[
\Delta_2(I_{m+\frac12})=kI_{m+\frac12}.
\]
Let
\[
\Delta_3=\Delta_2-k\tau,
\]
where $\tau$ is the outer derivation. Then
\[
\Delta_3(L_m)=\Delta_3(I_{m+\frac12})=0.
\]
By Lemma \ref{lemm-6}, we have $\Delta_3(\mathbf{l})=0$. Therefore $\Delta_3=0$, and hence
\[
\Delta=D_0-a_1\mathrm{ad}(L_0)+k\tau.
\]
Thus, $\Delta$ is a derivation.
\end{proof}
\end{thm}


\begin{thebibliography}{99}
        \bibitem{AKS}
        H. Abdelwahab, I. Kaygorodov and B. Sartayev, $\delta$-Poisson and transposed $\delta$-Poisson algebras, JoNAS. {\bf 1} (2026), no.~1. 

        \bibitem{AK}
        S.~A. Ayupov and K.~K. Kudaybergenov, Local derivations on finite-dimensional Lie algebras, Linear Algebra Appl. {\bf 493} (2016), 381--398.

       \bibitem{Bai} 
       C.~M. Bai et al., Transposed Poisson algebras, Novikov-Poisson algebras and 3-Lie algebras, J. Algebra {\bf 632} (2023), 535--566.

       \bibitem{B}
       K. Barron, On twisted modules for $N=2$ supersymmetric vertex operator superalgebras, in {\it Lie theory and its applications in physics}, 411--420, Springer Proc. Math. Stat., 36, Springer, Tokyo.
        \bibitem{BO}
       O.~Boranbaev, Local derivations on the twisted Heisenberg–Virasoro algebra, Asian-European J. Math. {\bf } (2026), 2650071 (13 pages).

        \bibitem{B&D}
        D. Burde and K. Dekimpe, Post-Lie algebra structures and generalized derivations of semisimple Lie algebras, Mosc. Math. J. {\bf 13} (2013), no.~1, 1--18, 189.

        \bibitem{C}
        C. Cheng et al., Quasi-Whittaker modules, J. Algebra {\bf 698} (2026), 288--315.

        \bibitem{CZZ}
        Y. Chen, K. Zhao and Y. Zhao, Local derivations on Witt algebras, Linear Multilinear Algebra {\bf 70} (2022), no.~6, 1159--1172.


        \bibitem{F98} 
        V.~T. Filippov, On $\delta$-derivations of Lie algebras, Siberian Math. J. {\bf 39} (1998), no.~6, 1218--1230.

        \bibitem{GMZ}
        D. Gao, Y. Ma and K. Zhao, Non-weight modules over the mirror Heisenberg-Virasoro algebra, Sci. China Math. {\bf 65} (2022), no.~11, 2243--2254.

        \bibitem{GZ}
        D. Gao and K. Zhao, Tensor product weight modules for the mirror Heisenberg-Virasoro algebra, J. Pure Appl. Algebra {\bf 226} (2022), no.~5, Paper No. 106929, 18 pp.
       
        

        \bibitem{Kad}
        R.~V. Kadison, Local derivations, J. Algebra {\bf 130} (1990), no.~2, 494--509.
      

       \bibitem{Kay}
       I. Kaygorodov, Non-associative algebraic structures: classification and structure, Commun. Math. {\bf 32} (2024), no.~3, 1--62.

       \bibitem{KKS}
       I. Kaygorodov, A. Khudoyberdiyev and Z. Shermatova, Quasi-derivations of Witt and related algebras, Res. Math. Sci. {\bf 13} (2026), no.~1, Paper No. 20, 23 pp.

       \bibitem{KJ}
       A.~K. Khudoyberdiyev and D. Jumaniyozov, Local derivations and automorphisms of nilpotent Lie algebras, Comm. Algebra {\bf 53} (2025), no.~5, 1921--1933.

       \bibitem{Lar}
       D.~R. Larson and A.~R. Sourour, Local derivations and local automorphisms of $B(X)$, in {\it Operator theory: operator algebras and applications, Part 2 (Durham, NH, 1988)}, 187--194, Proc. Sympos. Pure Math., 51, Part 2, Amer. Math. Soc., Providence, RI.
       
       \bibitem{LL}
       G.~F. Leger Jr. and E.~M. Luks, Generalized derivations of Lie algebras, J. Algebra {\bf 228} (2000), no.~1, 165--203.

       \bibitem{L}
       D. Liu et al., Irreducible modules over the mirror Heisenberg-Virasoro algebra, Commun. Contemp. Math. {\bf 24} (2022), no.~4, Paper No. 2150026, 23 pp.

       \bibitem{LC}
       H.~J. Liu and Z.~X. Chen, Biderivations of the mirror Heisenberg-Virasoro algebra, J. Algebra Appl. {\bf 24} (2025), no.~10, Paper No. 2550244, 20 pp.

      


       

       \bibitem{FKL} 
       B.~L. Macedo~Ferreira, I. Kaygorodov and V. Lopatkin, $\frac{1}{2}$-derivations of Lie algebras and transposed Poisson algebras, Rev. R. Acad. Cienc. Exactas F\'is. Nat. Ser. A Mat. RACSAM {\bf 115} (2021), no.~3, Paper No. 142, 19 pp.

       

        

        

       \bibitem{TYZ}
        H. Tan, Y.~F. Yao and K. Zhao, Classification of simple smooth modules over the Heisenberg-Virasoro algebra, Proc. Roy. Soc. Edinburgh Sect. A {\bf 155} (2025), no.~4, 1321--1365.

        \bibitem{WGL}
        Q. Wu, S.~L. Gao and D. Liu, Local derivations on the Lie algebra $W(2, 2)$, Linear Multilinear Algebra {\bf 72} (2024), no.~4, 631--643.

        

        \bibitem{XCK}
        C. Xu, $\delta$-biderivations of Virasoro related algebras, arXiv:2603.06005.

        \bibitem{YH}
        L. Yuan and Q. Hua, $\frac 12$-(bi)derivations and transposed Poisson algebra structures on Lie algebras, Linear Multilinear Algebra {\bf 70} (2022), no.~22, 7672--7701.

        \bibitem{ZC}
        Y. F. Zhao and Y. S. Cheng, Derivation algebra and automorphism group of the mirror Heisenberg-Virasoro algebra, Math. Theory Appl. {\bf 42} (2022), no.~4, 36--44.
	\end{thebibliography}
\end{document}